\documentclass[12pt,a4paper]{amsart}
\usepackage[latin1]{inputenc}
\usepackage[english]{babel}
\usepackage{amsmath,amstext,amsthm,amsfonts,amssymb,amscd}
\usepackage[dvips]{graphicx}
\usepackage{hyperref}

\usepackage[pagewise]{lineno}

\usepackage{cleveref}

\usepackage{color}

\theoremstyle{plain}

\newtheorem*{thA}{Theorem A}
\newtheorem*{thB}{Theorem B}
\newtheorem*{thC}{Theorem C}

\newtheorem{claim}{Claim}[section]

\newtheorem{definition}{Definition}[section]
\newtheorem{example}{Example}[section]

\newtheorem{lemma}{Lemma}[section]

\newtheorem{proposition}{Proposition}[section]
\newtheorem{remark}{Remark}[section]

\newcommand{\bc}{\begin{center}}
	\newcommand{\ec}{\end{center}}

\def\pp{{\mathbb{P}}}
\def\N{{\mathbb{N}}}

\def\Z{{\mathbb{Z}}}

\def\R{{\mathbb{R}}}

\def\||{\parallel}
\def\ds{\displaystyle}
\def\dfrac{\ds \frac}

\newcommand{\seta}{\rightarrow}

\newcommand{\Pp}{\mathbb{P}}

\newcommand{\E}{\mathcal{E}}

\begin{document}
	
	\title[Random Expansive Measures]{Random Expansive Measures}

	
	
	\author[R. A. Bilbao]{Rafael A. Bilbao}
	\address{Rafael A. Bilbao, Universidad Pedag\'ogica y Tecnol\'ogica de Colombia, Avenida
		Central del Norte 39-115, Sede Central Tunja, Boyac\'a, 150003, Colombia. 
	}
	\email{rafael.alvarez@uptc.edu.co}
	\author[M. Oliveira]{Marlon Oliveira}
	\address{Marlon Oliveira, Departamento de Matem\'atica e Inform\'atica, Universidade Estadual do Maranh\~ao,  65055-310  S\~ao Lu\'is, Brazil} \email{marlon.cs.oliveira@gmail.com}
	
	\author[E. Santana]{Eduardo Santana$^*$}
\address{Eduardo Santana, Universidade Federal de Alagoas, 57200-000 Penedo, Brazil}
\email{jemsmath@gmail.com}


\thanks{*Author to whom any correspondence should be addressed}

	\date{\today}

	
	
	\maketitle
	\begin{abstract}
		The notion of expansivity and its generalizations (measure expansive, measure positively expansive, continuum-wise expansive, countably-expansive) are well known for deterministic systems and can be a useful property for studying significant type of behavior, such as chaotic one. This study aims to extend these notions into a random context and prove a relationship between relative positive entropy and random expansive measures and apply it to show that if a random dynamical system has positive relative topological entropy then the $w$-stable classes have zero measure for the conditional measures. We also prove that there exists a probability measure that is both invariant and expansive. Moreover, we obtain a relation between the notions of random expansive measures and random countably-expansive systems.
	\end{abstract}

	\bigskip


	\section{Introduction}
	In the study of dynamical systems one of the objectives is to understand the structure of orbits and for that it is common to study certain properties. Among these we have the notion of expansivity, introduced by Utz in \cite{Utz} under the name of the unstable homeomorphism. Expansiveness means that the trajectories of nearby points are separated as the system evolves. This characterization is applicable to a large class of dynamical systems that exhibit chaotic behavior which highlights its importance. In the literature, there are many results related to expansivity, of which we cite some examples \cite{Eisenberg, Albert, Kifer2000, Vieitez, Williams1, Williams2}.
	
	Let $X$ be a metric space and $f: X \to  X$ a homeomorphism. By definition, we know that $f$ is said to be expansive if, and only if, there exists $\delta>0$ (the \textit{expansive constant})  such that $\Gamma_{\delta}(x) = \{x\}$ where 
	\[
	\Gamma_{\delta}(x) = \{y \in X : \, d(f^n(x), f^n(y))\leq \delta, \, \textrm{for every} \ n\in \Z \}
	\] 
	for all $x\in X$ see (\cite{Utz}). In \cite{Eisenberg} M. Eisenberg gave the definition of positive expansiveness that consider forward orbits on semigroups of endomorphisms.
	
	Due to the importance of the notion of expansivity, it has several generalizations that proved to be of great importance for the theory of dynamical systems. In recent years, Morales has extended the notion of expansivity for measurable spaces with  Borelian probability measures (see \cite{Morales}). We say that a Borel probability $\mu$ is \textit{expansive} if there exists $\delta>0$ such that $\mu(\Gamma_{\delta}(x)) =0$ for every $x\in X$. This notion leads to several results on dynamical systems. For example, in \cite{Pacifico} Armijo and Pac\'ifico show a relationship with the Lyapunov exponents. While in \cite{Morales-Sirvent} Morales and Sirvent  apply this concept to homemorphisms, presenting several examples. Other related results also can be found in \cite{Artigue-Carrasco} and \cite{LeeM}. \\
	Similarly to the case of positive expansivity, in \cite{Morales} Morales introduces the concept of positively measure-expansive systems as a generalization of the notion of positively expansive ones. The work \cite{ArbietroMorales} provides conditions for a measure to be expansive as a function of the entropy.

	It is important to look for simpler ways to understand interesting behaviors of a dynamical system and one of these ways is to ensure positive entropy, which indicates a chaotic behavior (see \cite{Li}). However, one of the most difficult tasks is to determine the entropy of a system, so it is relevant to study the presence of dynamical notions that imply positive entropy and, in this sense, it is already known that the existence of invariant expansive measures together with the shadowing property implies positive entropy (see  \cite{Lee-Park}, \cite{LeeM}, \cite{MoralesP} ). 
	
	In 1993, Kato \cite{Kato1993} defined expansivity by requiring that if a continuum has small diameter for all times then it is a singleton. Some interesting properties of continuum-wise expansivity have been obtained in \cite{Artigue-Carrasco}, \cite{Lee-Park}, \cite{Sakai}. The notion of countably-expansive systems is considered in \cite{Morales-Sirvent}, where the set $\Gamma_{\delta}(x)$ is countable for all $x \in X$. In \cite{Artigue-Carrasco} a relation between continuum-wise and countably-expansive systems is obtained.

	For bundle random dynamical systems defined as $F(w,x)=(\theta(w), f_{w}(x))$ where $\theta:\Omega \rightarrow \Omega$ is an invertible map preserving an ergodic measure $\mathbb{P}$ and $f_{w}: \E_w \rightarrow \E_{\theta(w)}$ is a continuous map with $\E_w \subset X$ and $\E_{w},X$ are compact spaces. The dynamics of the orbits are through the fibers $\E_w$ and therefore the information of the dynamics is concentrated on them. In this context  one can define the {\it relative entropy, relative topological entropy, variational principle, Lyapunov exponent} and other related concepts (see \cite{Arnold}).  At the moment, the study of expansiveness is still quite far from being well understood despite some advances in the area. Kifer and Gundlach in \cite{Kifer2000} introduced a notion of expansive random bundle transformation, which derives the uniqueness of equilibrium states. In view of so many interesting results related to the expansivity of deterministic systems, it is natural to ask whether the generalizations of the notion of expansivity extend to random dynamical systems and,  if so, whether it gives nice information about the dynamics of the random systems.
	
	In section 2, we introduce the definitions of random expansive measures, positively random expansive measures, random countably-expansive and random continuum-wise expansive systems and we give some examples in section \ref{Examples}. Additionally, we show a relationship between relative positive entropy and random expansive measures (see Theorem A and example \ref{exemploentropositiva1}) and we apply it to show that if a random dynamical system has positive relative topological entropy then the $w$-stable classes have zero measure with respect to  the conditional measures which generalizes the main results  \cite{ArbietroMorales}. Furthermore, we show the existence of a probability measure that is invariant and expansive in the random sense generalizing a result in \cite{MoralesP}. In these first two results, we also highlight some relevant properties of the measures in the random dynamical setup. In the end, we extend some results in \cite{Artigue-Carrasco} showing the relationship between a countably expansive map and a random expansive measure.\\
	
	\textbf{Acknowledgements:} The authors are grateful to Carllos Eduardo Holanda for his contribution in revising the manuscript and to the anonymous referees for their valuable comments which contributed to correct and improve our manuscript.
	

	\section{Preliminaries and Main Results}
	
	
	Let $(\Omega,\mathcal{F},\mathbb{P})$ be a Polish space together with a $\mathbb{P}$-preserving invertible map $\theta:\Omega \to \Omega$ and a  compact metric space $X$ with distance $d$ and denote by $\mathcal{B}$ its Borel $\sigma$-algebra. Let $\E$ be a measurable subset of $\Omega \times X$  with respect to the product $\sigma$-algebra $\mathcal{F}\times \mathcal{B}$ and the fibers $\E_{w}=\{x\in X ; \, (w,x)\in  \E\}$ of $\E$ be compact. A continuous bundle random dynamical system (RDS) $f= \{f_w\}_w$ over $(\Omega, \mathcal{F}, \pp, \theta)$ is generated by mappings $f_{w} :\E_{w} \seta \E_{\theta(w)}$ such that the map $(w,x) \to f_w(x)$ is measurable and the map $x \to f_w(x)$ is continuous for  $\pp$-almost all $w$. We will call a map $F: \E \seta\E$  defined by $F(w,x) = (\theta(w), f_{w}(x))$ as a skew product transformation. Observe that $F^{n}(w,x) = (\theta^n(w),f^n_{w}(x))$ where $f_{w}^{n} := f_{\theta^{n-1} (w)} \circ f_{\theta^{n-2} (w)} \circ \dots \circ f_{w}$ for $n\geq0$ , $f_{w}^{0} := Id$, and if $f_w$, for $w\in \Omega$, are homeomorphisms, $f^{n}_w := (f_{\theta^{n}(w)})^{-1}\circ \dots \circ (f_{\theta^{-1}(w)})^{-1}$ for $n\leq -1$, where $f_{w}^{-1} = (f_w)^{-1}$. For more details to see \cite{Arnold}.

	
	Let $\epsilon_\Omega$ be a partition of $\Omega$ into singletons and consider $\pi_{\Omega}: \Omega\times X \rightarrow \Omega$ the projection on $\Omega$. In this case, notice that $\pi^{-1}_{\Omega}(\epsilon_\Omega)$ is the partition whose elements are in $X$.
	
	Denote $\mathcal{P}_{\mathbb{P}}(\E)=\{ \mu \in \mathcal{P}_{\mathbb{P}}(\Omega\times X) ; \, \mu(\E)=1 \}$, where $\mathcal{P}_{\mathbb{P}}(\Omega\times X)$ is the space of probability measures on $(\Omega  \times X, \mathcal{F} \times \mathcal{B})$ having the marginal $\pp$ on $\Omega$ i.e. $\mu \circ \pi^{-1}_{\Omega}= \mathbb{P}$. The partition $\epsilon_{\Omega}$ is measurable and by Rokhlin's Disintegration Theorem (see \cite{Rokhlin}), each measure $\mu \in \mathcal{P}_{\mathbb{P}}(\E)$ can be factorized as
	$d\mu(w,x) = d\mu_{w}(x)d\mathbb{P}(w)$ where $\mu_{w}$ are regular conditional probabilities with respect to the $\sigma$-algebra $\mathcal{F}_{\E}$ formed by all sets $(A\times X)\cap \E$ with $A\in \mathcal{F}$. This means that $\mu_{w}$ is a probability measure on $\E_{w}$ for $\mathbb{P}$-a.e $w \in \Omega$ and for any measurable set $R\subset \E$ and $R_{w}=\{x ; \, (w,x)\in R \}$ we have $\mu (R) = \int \mu_{w}(R_{w}) d\mathbb{P}(w)$.   
	
	Denote by $\mathcal{M}(\E,f)$ the space of all $F$-invariant probabilities in $\mathcal{P}_{\mathbb{P}}(\E)$, which under our conditions is non-empty (see \cite{Arnold}, pages 36 and 43). A measure $\mu$ is $F$-invariant if and only if the disintegration of $\mu$ satisfy  $\mu_{w}\circ (f_{w})^{-1}  = \mu_{\theta(w)}$ for $\mathbb{P}$-a.e $w \in \Omega$. We say that a measure $\mu \in \mathcal{P}_{\pp}(\E)$ has non-atomic factorization if the conditional probability measures $\mu_w$ are non-atomic $\pp$-a.e $w\in \Omega$. A measure $\mu \in \mathcal{P}_{\mathbb{P}}(\E)$ is called ergodic if $F$ is an ergodic transformation.

	We call $\delta:\Omega \to (0,+\infty)$ a random variable. Let $m,n\in \Z$ be numbers such that $m<n$ and introduce a family of metrics on each fiber $\E_{w}$ as 
	\[
	d^{w}_{\delta; \, m,n}(x,y) := \max_{m\leq k < n}\left\lbrace d(f^{k}_{w}(x),f^{k}_{w}(y)) (\delta(\theta^{k}(w)))^{-1} \right\rbrace,
	\]
	for $x,y\in \E_{w}$. We denote $d^{w}_{\delta; n} = d^{w}_{\delta; \, 0,n}$ and $d^{w}_{\delta; \pm n} = d^{w}_{\delta; \, -n,n}$. Let $B_{w}[x,\delta,n]$ and $B_{w}[x,\delta,\pm n]$ be the closed balls in $\E_{w}$ centered at $x$ and having radius $1$ with respect to the metrics $d^{w}_{\delta; \, n}$ and $d^{w}_{\delta; \, \pm n}$, that is, 
	\[
	B_{w}[x,\delta,n] = \{y\in \E_{w}; \ d(f^{k}_{w}(x),f^{k}_{w}(y))\leq \delta(\theta^{k}(w)), \  0\leq k \leq n-1  \},
	\]
	and
	\[
	B_{w}[x,\delta,\pm n] = \{y\in \E_{w}; \ d(f^{k}_{w}(x),f^{k}_{w}(y))\leq \delta(\theta^{k}(w)), \ -n + 1 \leq k \leq n-1  \}.
	\]
	We write $B_{w}[x,\delta] := B_{w}[x,\delta,1]$ and denote $\Gamma_{\delta}(x,w) = \bigcap_{n\geq 1} B_{w}[x,\delta,\pm n] $ and $\Gamma_{\delta}^{+}(x,w) = \bigcap_{n\geq 1} B_{w}[x,\delta, n]$. Since they are non-empty closed sets on the fibers, we know that they are measurable.  For more information, please refer to \cite{KiferLiu}.
	
	
	\bigskip
	
	\bigskip
	
	From now on, we write RDS to refer to a continuous bundle random dynamical system $f= \{f_w\}_w$ as above.
	
	\subsection{Random Expansive Measures}\label{RDS}  
	
	In \cite{KiferLiu}, we have the following definition: 
	\begin{definition}\label{expansive}
		A RDS $f= \{f_w\}_w$ is called relative expansive if there exists a random variable $\delta:\Omega \to (0,+\infty)$, called expansivity characteristic, such that for $\pp$-a.e $w$ we have $\Gamma_{\delta}(x,w) = \{x\}$ for all $x\in \E_{w}$. 
	\end{definition}
	
	Now, we introduce a notion of expansive measures in the random sense.
	\begin{definition}
		\label{expansive measure} 
		Let $f= \{f_w\}_w$ be a RDS. We say that a Borel probability measure $\mu \in \mathcal{P}_{\mathbb{P}}(\E)$ is random expansive (resp. positively random expansive) for $F$ if there exists a random variable  $\delta:\Omega \to (0,+\infty)$, called expansivity characteristic, such that for $\pp$-a.e $w$ occurs $\mu_{w}(\Gamma_{\delta}(x,w))=0$ (resp. $\mu_{w}(\Gamma_{\delta}^{+}(x,w))=0$) for all $x\in \E_{w}$. 
	\end{definition}
	
	\begin{remark} We make some remarks here:
		\begin{enumerate}
			\item Since a random variable is a class of measurable functions which coincide in a full measure set, once given a random variable $\delta$ and two functions $\rho, \tau$ in the same class $\delta$, the sets $\Gamma_{\rho}(x,w)$ and $\Gamma_{\tau}(x,w)$ (and $\Gamma^{+}_{\rho}(x,w)$ and $\Gamma^{+}_{\tau}(x,w)$, respectively) may slightly differ on a set of zero measure and it does not affect our results. 
			\item We observe that for $\Omega = \{w_{0}\}$ the Definition  \ref{expansive} and Definition \ref{expansive measure} reduce to the notions of expansivity for deterministic cases as in \cite{Utz} and \cite{Morales}.
			\item Let us fix a relative expansive RDS $f= \{f_w\}_w$.  If a Borel probability $\mu$ is a random expansive measure, then the factorization of $\mu$ is non-atomic, i.e., $\mu_{w}$ is a non-atomic measure for $\mathbb{P}$-a.s. $w \in \Omega$. Conversely, if a Borel probability measure $\mu \in \mathcal{P}_{\mathbb{P}}(\E)$ is such that its factorization is non-atomic, then $\mu$ is random expansive measure. In fact, we have that for $\mathbb{P}$-a.s. $w \in \Omega$ and for every $x \in \E_{w}$ it holds that $\Gamma(x,w) = \{x\}$ and so $\mu_{w}(\Gamma(x,w))=\mu_{w}(\{x\})$.
		\end{enumerate}	
	\end{remark}
	
	Let $\mathcal{R} = \{R_i\}$ be a finite or countable partition of $\E$ into measurable sets then $\mathcal{R}(w) = \{ R_i(w)\}$, where $R_i(w) = \{ x\in \E_w : \, (w,x)\in R_i\}$ is a partition of $\E_w$. Given a $\sigma$-algebra $\mathcal{F}_{\E}$ over $\E$, for $\mu \in \mathcal{M}(\E,f)$ the conditional entropy of $\mathcal{R}$,  is defined by
	\[
	H_{\mu}(\mathcal{R}|\mathcal{F}_{\E}) = - \int \sum_{i}\mu (R_i| \mathcal{F}_{\E}) \log \mu (R_i|\mathcal{F}_{\E}) \, d\mu = \int H_{\mu_w}(\mathcal{R}(w)) \, d\pp
	\]
	where $H_{\mu_w}(\mathcal{R}(w))$ denotes the usual entropy of the partition $\mathcal{R}(w)$. The \textit{relative entropy} of the skew product $F$ or the RDS $f=\{f_w\}_w$, with respect to $\mu$,  which will be denoted by $h_{\mu}(F|\theta)$ or $h_{\mu}(f)$, respectively, is defined as
	\[
	h_{\mu}(F|\theta) = h_{\mu}(f):= \sup_{\mathcal{R}} \{h_{\mu} (f, \mathcal{R})\}
	\]
	where 
	\[
	h_{\mu} (f,\mathcal{R}) = \lim_{n\seta \infty} \dfrac{1}{n} H_{\mu} \left( \bigvee_{i=0}^{n-1} (F^{i})^{-1} (\mathcal{R} | \mathcal{F}_{\E}) \right) 
	\]
	is called the relative entropy of $f$ with respect to a partition $\mathcal{R}$ and $\bigvee$ denotes the join of partitions. The supremum is taken over all finite measurable partitions $\xi$ of $X$. The reader can find more results and properties in Kifer \cite{KiferLiu}.
	
	For each $w\in \Omega$ and $p\in \E_{w}$, we define the
	\textit{w-stable set} of $p$ by  
	\[
	W_{w}^{s}(p)=\{ x\in \E_{w}; \, \lim_{n\seta +\infty} d(f_{w}^{n}(x),f_{w}^{n}(p))=0\} .
	\]
	
	The \textit{w-stable class} is a subset equal to $W_{w}^{s}(p)$ for some $p\in \E_{w}$. 
	
	\bigskip
	
	We are able to state our first main result, which guarantees that positive entropy implies the existence of positively random expansive measures and $w$-stable classes with zero measures in the fibers.
	
	\begin{thA}
		\label{TeoA}
		Let $f=\{ f_w\}_w$ be a RDS and $\mu \in \mathcal{M}(\E,f)$ be an ergodic measure. If $h_{\mu}(f) > 0$, then the following holds 
		\begin{itemize}
			\item[(1)] the measure $\mu$ is positively random expansive.
			\item[(2)] For $\Pp$-a.e $w$ the \textit{w-stable classes} of $F$ have measure zero with respect to $\mu_{w}$ the conditional measures of $\mu$ .
		\end{itemize}

	\end{thA}

	The next result gives the existence of a probability measure that is both invariant and random expansive .
	
	\begin{thB}
		\label{thc} 
		Let $f=\{ f_w\}_w$ be a RDS. If there exists a random expansive measure $\mu\in\mathcal{P}_{\mathbb{P}}(\E)$ for $f$, then there exists a measure $\mu'\in \mathcal{M}(\E,f)$ that is random expansive. 
	\end{thB}
	The proof of the above theorem follows the ideas of \cite{Morales-Sirvent}, but with some appropriate modifications, necessary to guarantee the measurability of the desintegration.

	\subsection{Random countably expansive and continuum-wise expansive} \   \    \    \    \    \    \    \    \  
	
	In this section we introduce other notions of expansivity for RDS that generalize those of \cite{Morales-Sirvent} and \cite{Kato1993} and present relation between them. 
	
	
	\medskip
	\begin{definition}
		Let  $f=\{ f_w\}_w$ be a RDS. We say that $f$ is {\it random countably-expansive} if there exist a random variable $\delta:\Omega \to (0,+\infty)$, called expansivity characteristic, such that for $\pp$-a.e $w$ the set $\Gamma_{\delta}(x,w)$ is countable for all $x\in \E_{w}$. 	
	\end{definition}
	
	\begin{definition}
		Let  $f=\{ f_w\}_w$ be a RDS. We say that $f$ is random measure-expansive if all Borel probability measure $ \mu \in \mathcal{P}_{\mathbb{P}}(\E)$ with factorization non-atomic is a random expansive measure of $F$.
	\end{definition}
	
	A continuum set is a compact and connected set. A homeomorphism $T : X \to X$ is continuum-wise expansive if there exists $\delta > 0$ such that for every continuum set $C \subset X$ that is not a singleton, there exists $n \in \mathbb{Z}$ such that $\text{diam}(T^{n}(C)) > \delta$ (see \cite{Kato1993} for more details). 
	
	We proceed with the definition for the random context.
	
	\begin{definition}
		Let  $f=\{ f_w\}_w$ be a RDS. We say that $f$ is {\it random continuum-wise expansive} if there exist a random variable $\delta:\Omega \to (0,+\infty)$, called expansivity characteristic, such that for $\pp$-a.e. $w$ every continuum set $D_w \subset \E_{w}$, that is not a singleton, there exists $n \in \mathbb{Z}$ such that $\text{diam}(f_{w}^{n}(D_w)) > \delta(\theta^{n}(w))$.  
	\end{definition}
	
	In the following, we prove a similar result as in \cite{Artigue-Carrasco}.
	\begin{lemma}
		Given a RDS  $f=\{ f_w\}_w$, the following statements are equivalent:
		\begin{itemize}
			\item[(i)] $f$ is random continuum-wise expansive.
			
			\item[(ii)] There exists  a positive random variable $\delta$ such that for $\pp$-a.e. $w$ the set $\Gamma_{\delta}(x,w)$ contains only continuum set which are singleton for all $x\in \E_w$.
		\end{itemize}
	\end{lemma}
	\begin{proof}
		$(i) \, \Rightarrow \, (ii)$ 
		Let $\epsilon = \epsilon(w)>0$ be the expansivity characteristic of $f$ and $\delta = \epsilon/2$ be a random variable. Given a continuum set $C \subset \Gamma_{\delta}(x,w)$, we have $\text{diam}(f_{w}^{n}(C)) \leq 2\delta(\theta^n(w))$ for every $n \in \mathbb{N}$, then by random continuum-wise expansivity we conclude that $C$ is a singleton.
		
		$(ii) \, \Rightarrow \, (i)$ 
		We wish to show that $f$ is random continuum-wise expansive. Given  a continuum set $C\subset \E_{w}$, let us suppose that $\text{diam}(f_{w}^{n}(C)) \leq \delta(\theta^n(w))$ for every $n \in \mathbb{Z}$. Let us fix a point $x \in C$. For any $y\in C$ we have $d(f_{w}^{n}(x),f_{w}^{n}(y))\leq \delta(\theta^n(w))$ for every $n \in \mathbb{Z}$, which implies that $y \in \Gamma_{\delta}(w,x)$, so we have $C \subset \Gamma_{\delta}(x,w)$.  Since, by hypothesis $\pp$-a.e. $w$, the set $\Gamma_{\delta}(x,w)$ contains only singletons, so the continuum set $C$ is trivial. Therefore, $F$ is random continuum-wise expansive.
	\end{proof}
	

	When a continuum is not a single point it must be uncountable. In order to see this, let $C$ be the continuum and fix $x_{0} \in C$. Define the function $g: C \to \mathbb{R}$ as $g(x) = d(x,x_{0})$. Since $g$ is a continuous function and $C$ is a connected set, we have that the image $g(C)$ is a connected subset of $\mathbb{R}$. So, it is an interval and we have that $g(C)$ is either a single point or uncountable. Then, $C$ is also either a single point or uncountable. By this fact we readily obtain the next result.
	
	\begin{proposition}
		For a RDS $f=\{ f_w\}_w$, we obtain the following
		\[
		\text{random expansive} \implies \text{random countably-expansive} \implies \text{random continuum-wise expansive}
		\] 
	\end{proposition}
	
	

	In the sequence, we present a theorem that generalizes the result for deterministic systems contained in \cite{Artigue-Carrasco}. For the next result, we require that $\E = \Omega \times X$.

	\begin{thC}
		\label{teoremaB}
		Let $f=\{ f_w\}_w$ be a RDS generated by $f_w : X \to X$ for all $w\in \Omega$. The following statements are equivalent:
		\begin{itemize}
			\item [(i)] $f$ is random countably-expansive.
			\item [(ii)] $f$ is random measure-expansive.
		\end{itemize}
	\end{thC}
	
	In the next sections, we present the the proofs of our main results.
	
	\section{Proof of Theorem A}
	
	We begin this section by proving a result which weaken the condition of random expansivity over the fibers. After that, we divide the proof of Theorem A into some lemmas. 
	\begin{lemma}\label{lema1}
		A Borel measure $\mu\in\mathcal{P}_{\pp}(\E)$ is random expansive measure for the RDS $f=\{ f_w\}_w$ if and only if there exists a positive random variable $\delta_{0}$ on $\Omega$, such that for $\mathbb{P}$-a.e $w\in \Omega$ the following holds: $\mu_{w}(\Gamma_{\delta_{0}}(x,w)) = 0$ for $\mu_{w}$-a.e $x\in \E_{w}$.
	\end{lemma}
	
	\begin{proof}
		Only the ``if'' part needs to be shown. In the fiber $\E_w$ we have the measurable subset
		$$
		C_{w}=\{y\in \E_w | \ \mu_{w}(\Gamma_{\delta_0}(y,w))>0 \}.
		$$
		By hypothesis, we have $\mu_{w}(C_{w})=0$ for $\mathbb{P}$-a.e $w\in \Omega$. Therefore, setting 
		$$
		\Tilde{\Omega}=\{w\in \Omega \ | \ \mu_{w}(C_{w})=0\}  
		$$
		it is so that $\Tilde{\Omega}$ is a measurable set and $\mathbb{P}(\Tilde{\Omega})=1$. 
		We claim that the measure $\mu$ is random expansive with expansivity characteristic $\delta = \delta_{0}/2$. In order to do this, let $w_{0} \in \Tilde{\Omega}$ and we will show that $\mu_{w_{0}}(\Gamma_{\delta_{0}/2}(y,w_{0})) = 0$ for every $y \in \E_{w_{0}}$. We know that for $z\in \Gamma_{\delta_{0}/2}(y,w_{0})$, we have that $			\Gamma_{\delta_{0}/2}(y,w_{0}) \subset  \Gamma_{\delta_{0}}(z,w_{0}).$
		Therefore, by contradiction, if we had  $\mu_{w_{0}}(\Gamma_{\delta_{0}/2}(y,w_{0}))>0$, since $\mu_{w_{0}}(C_{w_{0}}) = 0$ we could find $z \in \Gamma_{\delta_{0}/2}(y,w_{0}) \backslash C_{w_{0}}$. It means that
		\[
		0 < \mu_{w_{0}}(\Gamma_{\delta_{0}/2}(y,w_{0})) \leq  \mu_{w_{0}}(\Gamma_{\delta_{0}}(z,w_{0})) = 0.
		\]
		This contradiction shows that $\mu_{w_{0}}(\Gamma_{\delta_{0}/2}(y,w_{0}))=0$. By arbitrariness of $w_{0} \in \Tilde{\Omega}$ and $z \in \Gamma_{\delta_{0}/2}(y,w_{0})$ we conclude that $\mu_{w_{0}}(\Gamma_{\delta_{0}/2}(y,w_{0})) = 0$ for every $y \in \E_{w_{0}}$, $\mathbb{P}$-a.e. $w \in \Omega$. Hence, $\mu$ is a random expansive measure with expansivity characteristic $\delta = \delta_{0}/2$.
	\end{proof}
	
	Now, we present a necessary and sufficient condition for a measure to be expansive in the random sense.
	
	

	\begin{lemma}\label{lema2'}
		A Borel measure $\mu\in\mathcal{P}_{\pp}(\E)$ is random expansive measure for a RDS $f=\{ f_w\}_w$ if and only if there exist a positive random variable $\delta$ on $\Omega$ and measurable sets $B_{\delta} \subset \E_{\delta}$ such that $\mu(B_{\delta}) = 1$, where
		\[
		\E_{\delta} =\{(x,w)\in \E; \, \mu_{w}(\Gamma_{\delta}(x,w))=0\}.
		\]
	\end{lemma}
	\begin{proof}
		Our goal is to prove that $\mu$ is a random expansive measure. By considering $\E \setminus B_{\delta}=\{(w,x): \ \mu_{w}(\Gamma_{\delta}(x,w)) > 0 \}$ and $\pi_{\Omega}:\E \rightarrow \Omega$ the projection on $\Omega$, then we claim that we have $\mathbb{P}(\pi_{\Omega}(\E\setminus B_\delta) \backslash \pi_{\Omega}(B_\delta))=0$. In fact, we have that
		
		
		\begin{align*}
			0=\mu(\E\setminus B_\delta)&= \int_{\pi_{\Omega}(\E\setminus B_\delta)}\mu_{w}(\E_{w}\cap( \E\setminus B_{\delta}))d\mathbb{P} \\
			&\geq \int_{\pi_{\Omega}(\E\setminus B_\delta) \backslash \pi_{\Omega}(B_{\delta})}\mu_{w}(\E_{w}\cap( \E\setminus B_{\delta}))d\mathbb{P}\\
			&=\int_{\pi_{\Omega}(\E\setminus B_\delta) \backslash \pi_{\Omega}(B_{\delta})} 1 \ d\mathbb{P}\\
			& = \mathbb{P}(\pi_{\Omega}(\E\setminus B_\delta)\setminus \pi_{\Omega}(B_\delta)) \geq 0,
		\end{align*}
		
		where for $w\in \pi_{\Omega}(\E\setminus B_\delta) \backslash \pi_{\Omega}(B_\delta)$ we have that $\E_{w}\cap (\E \setminus B_\delta)= \E_w$ and then  $\mu_{w}(\E_{w}\cap (\E \setminus B_\delta))=1$. This implies that $\mathbb{P}(\pi_{\Omega}(\E\setminus B_\delta) \backslash \pi_{\Omega}(B_\delta))=0$ and there exists $\widetilde{\Omega}\subset \Omega$ with $\mathbb{P}(\widetilde{\Omega})=1$, such that for all $w\in \widetilde{\Omega}$ we obtain 
		$$
		\mu_{w}(\Gamma_{\delta}(x,w))=0, \ \ \text{for all} \ \ x\in \E_w.
		$$
		Conversely, if $\mu$ is random expansive measure for $f=\{ f_w\}_w$, by definition  there exists a set $\widetilde{\Omega}\subset \Omega$ with $\mathbb{P}(\widetilde{\Omega}) = 1$ and a positive random variable $\delta$, such that for all $w\in\widetilde{\Omega}$ occurs $\mu_{w}(\Gamma_{\delta}(x,w))=0$ for all $x\in \E_{w}$. 
		
		Then, $\E_{\delta} \supset (\tilde{\Omega} \times X) \cap \mathcal{E}$ and $\mu ((\tilde{\Omega} \times X) \cap \mathcal{E}) = 1$. In fact, we have that
		\[
		\mu((\tilde{\Omega} \times X) \cap \E) = \int \mu_{w}((\tilde{\Omega} \times X) \cap \mathcal{E} \cap \E_{w}) d \mathbb{P} = \int_{\tilde{\Omega}} \mu_{w}(\E_{w}) d\mathbb{P} = \int_{\tilde{\Omega}} 1 d \mathbb{P} = \mathbb{P}(\tilde{\Omega}) = 1.
		\]
		It is enough to take $B_{\delta}$ as $(\tilde{\Omega} \times X) \cap \mathcal{E}$. The lemma is proved.	 
		
	\end{proof}
	
	The same results can be obtained for the case with positively random expansive measures. 
	

	\bigskip
	
	Now, generalizing the deterministic case, we prove the following:

	\begin{lemma}\label{lema3}
		Let $f=\{ f_w\}_w$ be a RDS. If $\mu \in\mathcal{M}(\E,f)$ is positively random expansive measure for $f=\{ f_w\}_w$, then the  \textit{w-stable classes} of $f$ have zero measure for $\mathbb{P}$-a.e.  $\mu_{w}$ conditional measures of $\mu$.
	\end{lemma}
	\begin{proof}
		Let $\{\gamma_i\}_{i\geq 0}$ be a monotone sequence of positive random variables on $\Omega$, such that  $\gamma_{i}\rightarrow 0$ uniformly. We can describe the $w$-stable set of a point $x\in \E_{w}$ by
		\[
		W^{s}_{w}(x) = \bigcap_{i\in \N^{+}}\bigcup_{j\in \N}\bigcap_{k\geq j} f^{-k}_{w}(B_{w}[f_{w}^{k}(x),\gamma_{i}]).
		\]
		Since $i\in \N^{+}$ $\gamma_{i+1}\leq  \gamma_{i}$ for all $i \in \N^{+}$, so
		\[
		\bigcup_{j\in N}\bigcap_{k\geq j} f_{w}^{-k}(B_{w}[f^{k}_{w}(x),\gamma_{i+1}]) \subseteq \bigcup_{j\in N}\bigcap_{k\geq j} f_{w}^{-k}(B_{w}[f^{k}_{w}(x),\gamma_{i}]) 
		\]
		Thus, we obtain 
		\[
		\mu_{w}(W^{s}_{w}(x)) \leq  \lim_{i \seta \infty} \sum_{j\in \N} \mu_{w} \left( \bigcap_{k\geq j}  f_{w}^{-k} \left(  B_{w}[f^{k}_{w}(x),\gamma_{i}]\right) \right)  .  
		\]
		Notice that for every $j \in \N$, occurs
		\[
		\cap_{k\geq j} f^{-k}_{w}(B[f_{w}^{k}(x),\gamma_{i}]) = f^{-j}_{w}(\Gamma^{+}_{\gamma_{i}}(f^{j}_{w}(x), \theta^{j}(w))).
		\]
		By $F$-invariance of $\mu$ we have 
		\begin{eqnarray}
			\mu_{w}\left( \cap_{k\geq j} f^{-k}_{w}(B[f_{w}^{k}(x),\gamma_{i}])\right)  & = & \mu_{w}\left(f^{-j}_{w}(\Gamma^{+}_{\gamma_{i}}(f^{j}_{w}(x),\theta^{j}(w))) \right)  \nonumber \\
			& = & \mu_{\theta^{j}(w)}(\Gamma^{+}_{\gamma_{i}}(f^{j}_{w}(x),\theta^{j}(w))).  \nonumber	
		\end{eqnarray}
		For values of $i$  large enough such that $\gamma_{i} < \delta_{0}$, where $\delta_{0}$ is the characteristic of random expansiveness of $\mu$, then since $\theta$ is invertible we have    $\mu_{\theta^{j}(w)}(\Gamma_{\gamma_{i}}(f^{j}_{w}(x),\theta^{j}(w))) = 0$ for $\mathbb{P}$-a.e. $w\in \Omega$. Therefore $\mu_{w}\left( 	\cap_{k\geq j} f^{-k}_{w}(B[f_{w}^{k}(x),\gamma_{i}])\right) =0$ and we conclude the proof. 
		
	\end{proof}

	With the previous results, we are able to prove Theorem A.

	\begin{proof}[Proof of Theorem A]
		Given a positive random variable $\delta$ on $\Omega$, we consider the set defined in Lemma \ref{lema2'}
		\[
		\E_{\delta} = \{(w,x)\in \E ; \, \mu_{w}(\Gamma_{\delta}(x,w))= 0 \}	
		\]
		where $\{\mu_{w}\}$ are the conditional measures of $\mu$, an ergodic invariant measure such that $0 < h_{\mu}(F|\theta) < +\infty$.
		Our goal is to show that there exist measurable sets $B_{\delta} \subset \E_{\delta}$ such that $\mu(B_{\delta}) = 1$. Then, by Lemma \ref{lema2'} we can see that $\mu$ is random expansive. 

		Fix a random variable $\delta>0$. We define the random map $\varphi_{\delta}: \E \seta \R\cup\{+\infty\}$ given by 
		\[
		\varphi_{\delta}(w,x) = \liminf_{n\in \N}\dfrac{-\log\mu_{w}(B_{w}[x,\delta,n])}{n},
		\]
		where $B_{w}[x,\delta,n] =  \cap_{j=0}^{n-1}f_{w}^{-j}\left( B_{\theta^{j}(w)}[f_{w}^{j}(x),\delta] \right)$. We remind that $B_{w}[x,r]$ denotes the closed ball with radius $r$ around $x$ in $\E_{w}$.
		Let $\delta_{k}$ be a sequence of positive random variables such that $\delta_{k} \seta 0$ when $k \seta +\infty$. Define $H=h_{\mu}(f)/2$ and the sets 
		\[
		\E^{k}= \{(w,x)\in\E; \,  \varphi_{\delta_{k}}(w,x)> H  \}
		\]
		for all $k\in \mathbb{N}$. 
		
		We claim that $F^{-1}(\E^{k}) \subset \E^{k}$ for every $k \in \mathbb{N}$, that is, $F^{-1}((w,x)) \in \E^{k}$ if $(w,x) \in \E^{k}$. It is equivalent to showing that
		\[
		\varphi_{\delta_{k}}(\theta^{-1}(w),f_{\theta^{-1}(w)}^{-1}(x))> H \,\, \textrm{if} \,\, \varphi_{\delta_{k}}(w,x)> H.
		\]
		
		By definition, if we write $z=\theta^{-1}(w)$, we have that
		\[
		\varphi_{\delta}(z,f_{z}^{-1}(x)) = \liminf_{n\in \N}\dfrac{-\log\mu_{z}(B_{z}[f_{z}^{-1}(x),\delta,n])}{n},
		\]
		where $B_{z}[f_{z}^{-1}(x),\delta,n]) =  \cap_{j=0}^{n-1}f_{z}^{-j}\left( B_{\theta^{j}(z)}[f_{z}^{j}(f_{z}^{-1}(x)),\delta] \right)$.
		
		Since $\mu$ is an invariant probability, we obtain
		\[
		\mu_{w}(B_{w}[x,\delta,n]) = \mu_{z}(f_{z}^{-1}(B_{w}[x,\delta,n])).
		\]
		But
		\[
		f_{z}^{-1}(B_{w}[x,\delta,n]) = f_{z}^{-1}(\cap_{j=0}^{n-1}f_{w}^{-j}\left( B_{\theta^{j}(w)}[f_{w}^{j}(x),\delta] \right)) = \cap_{j=0}^{n-1}f_{z}^{-1}(f_{w}^{-j}\left( B_{\theta^{j}(w)}[f_{w}^{j}(x),\delta] \right)) \supset
		\]
		\[
		\supset \cap_{j=0}^{n-1}f_{z}^{-1}(f_{w}^{-j}\left( B_{\theta^{j}(w)}[f_{w}^{j}(x),\delta] \right)) \cap B_{z}[f_{z}^{-1}(x)),\delta] = B_{z}[f_{z}^{-1}(x),\delta,n]).
		\]
		
		Then, we obtain
		\[
		\mu_{z}(B_{z}[f_{z}^{-1}(x),\delta,n]) \le \mu_{z}(f_{z}^{-1}(B_{w}[x,\delta,n])) = \mu_{w}(B_{w}[x,\delta,n]),
		\]
		which implies that
		\[
		-\log \mu_{z}(B_{z}[f_{z}^{-1}(x),\delta,n]) \ge -\log \mu_{w}(B_{w}[x,\delta,n]),
		\]
		that is,
		\[
		\varphi_{\delta}(z,f_{z}^{-1}(x)) = \liminf_{n\in \N}\dfrac{-\log\mu_{z}(B_{z}[f_{z}^{-1}(x),\delta,n])}{n} \ge
		\]
		\[
		\ge \liminf_{n\in \N}\dfrac{-\log\mu_{w}(B_{w}[x,\delta,n])}{n} = \varphi_{\delta}(w,x) > H. 
		\]
		
		It means that $F^{-1}(w,x) = (z,f_{z}^{-1}(x)) \in \E^{k}$, if $(w,x) \in \E^{k}$ and we obtain $F^{-1}(\E^{k}) \subset \E^{k}$, as claimed. Once $\mu$ is ergodic, we know that $\mu(\E^{k}) = 0$ or $\mu(\E^{k}) = 1$. 
		
		Now, note that for all $(w,x)\in \E$, $\varphi_{\delta}(w,x)$  is decreasing on $\delta$, so $\E^{k_1} \subset \E^{k_2}$ whenever $k_1<k_2$. 
		As a result, by taking $\delta$ as a random variable, we have 
		\[
		\{(w,x)\in \E; \ \sup_{\delta>0} \varphi_{\delta}(w,x) \geq h_{\mu}(F) \} \subset \cup_{k\in \N} \E^{k}.
		\]
		Then $ \mu (\{(w,x)\in \E; \ \sup_{\delta>0} \varphi_{\delta}(w,x) \geq h_{\mu}(F) \}) \leq \lim_{k\seta \infty} \mu(\E^{k}).$
		
		\begin{claim}\label{BK}
			\[\mu (\{(w,x)\in \E; \ \sup_{\delta>0} \varphi_{\delta}(w,x) \geq h_{\mu}(f) \}) = 1.
			\]	
		\end{claim}
		\begin{proof}
			By using the random Brin-Katok formula (see \cite{Zhu}), we obtain 
			\[\mu (\{(w,x)\in \E; \ \sup_{\epsilon>0} \varphi_{\epsilon}(w,x) = h_{\mu}(f) \}) = 1,
			\]
			where $\epsilon$ is taken as constant random variables. By fixing $(w,x) \in \E$ and taking a random variable $\delta_{k} \leq \epsilon$ (uniformly for $k$ sufficiently large), we have that $\varphi_{\delta_{k}}(w,x) \geq \varphi_{\epsilon}(w,x)$. It implies that
			\[
			\sup_{\delta>0} \varphi_{\delta}(w,x) \geq \sup_{\epsilon>0} \varphi_{\epsilon}(w,x) = h_{\mu}(f). 
			\]
			$\mu$-almost everywhere, because the supremum over all random variables $\delta$ is bigger than the supremum over all constant random variables $\epsilon$. This means that
			\[
			\{(w,x)\in \E; \ \sup_{\delta>0} \varphi_{\delta}(w,x) \geq h_{\mu}(F) \} \supset \{(w,x)\in \E; \ \sup_{\epsilon>0} \varphi_{\epsilon}(w,x) = h_{\mu}(f) \},  
			\]
			that is,
			\[\mu (\{(w,x)\in \E; \ \sup_{\delta>0} \varphi_{\delta}(w,x) \geq h_{\mu}(F) \}) = 1.
			\] 
		\end{proof}
		
		Thus, by using of Claim \ref{BK}, we obtain $\lim_{k\seta \infty}\mu(\E^{k}) = 1$. Let $k_{0}\in \N$ such that $\mu (\E^{k_{0}}) = 1$. Given a random variable $\delta \leq \delta_{k_{0}}$ and a point $(w,x)\in \E^{k_{0}}$, so $\varphi_{\delta_{k_{0}}}(w,x)>H$. 
		Thus 
		\[\liminf_{n\in \N}\dfrac{-\log \mu_{w}(B_{w}[x,\delta_{k_{0}},n])}{n} >H,
		\]
		then there exist $n_{0}\in \N$ such that all $n\geq n_{0}$ occurs  
		\[
		\dfrac{-\log \mu_{w}(B_{w}[x,\delta_{k_{0}},n])}{n}>H,
		\]
		so $\mu_{w}(B_{w}[x,\delta_{k_{0}},n]) <e^{-n.H}$.
		Thus, when $n\seta +\infty$ we  conclude that 
		\[
		\lim_{n\seta +\infty} \mu_{w}(B_{w}[x,\delta_{k_{0}},n]) = 0,
		\]
		so it implies that 
		$(w,x) \in \E_{\delta_{k_{0}}} $. By arbitrariness of $(w,x)\in \E^{k_{0}}$ we obtain that $\E^{k_{0}} \subset \E_{\delta_{k_{0}}}$.  By taking $B_{\delta_{k_{0}}} = \E^{k_{0}}$, we have completed the proof of the first item of the theorem.
		
		To verify the second item, it is enough to note that due to item (1), $\mu$ is a positively random expansive measure, so by Lemma\ref{lema3} we have that the $w$-stable classes have zero measure for $\mathbb{P}$-a.s. $\mu_{w}$ conditional measures of $\mu$.  
	\end{proof}
	\begin{remark}
		In the case of the bundle (RDS) $f = \{f_{w}\}_{w}$ being such that the functions $f_{w}: \E_{w} \to \E_{\theta(w)}$ are homeomorphisms, we have the system invertible and we can consider a type of $w$-unstable class $W^{u}_{w}(x)$, which we can prove analogously to have zero measure $\mathbb{P}$-a.s. $\mu_{w}$ conditional measures of $\mu$. In fact, we have that $\Gamma(x,w) \subset \Gamma^{+}(x,w)$.
	\end{remark}
	
	\section{Proof of Theorem B}
	
	We are going to divide the proof of Theorem B into three lemmas, the first of which consists of the invariance of the set $\Gamma_\delta (x,w)$. The second lemma starts from an expansive measure and using the pullback in each fiber, a new system of expansive measures is constructed. Finally, we define a sequence of measures in each fiber, then by the compactness of the fibers, we show the existence of a measurable and invariant measure of disintegration. In this part let us denote the orbits of $w$ through $\theta$ as $\theta^{n}(w)= w_n$ for all $n\in \mathbb{Z}$ considering $w_0= w$.
	\begin{lemma}
		\label{lemma1thc}
		The sets $\Gamma_\delta (x,w)$ are invariant with respect to the dynamics on the fibers for all $(w,x)\in \E$ i.e.  
		\[
		(f_{w_{-1}})^{-1}(\Gamma_\delta (x,w)) = \Gamma_{\delta} ( (f_{w_{-1}})^{-1}(x),w_{-1}).
		\]
	\end{lemma}
	
	\begin{proof}
		Let $y\in (f_{w_{-1}})^{-1}(\Gamma_\delta (w,x))$, then $f_{w_{-1}}(y) \in \Gamma_\delta (w,x)$, therefore
		\[
		d(f^{k}_{w}(x),f_{w}^{k}(f_{w_{-1}}(y)) ) < \delta_{w_{k}}
		\]
		For $k=0$ we have
		\[
		d(x, f_{w_{-1}}(y)) < \delta_w,
		\]
		which corresponds to $k=1$ in $\Gamma_{\delta} (w_{-1}, (f_{w_{-1}})^{-1}(x))$. Indeed 
		\[
		d(f_{w_{-1}}(y), f_{w_{-1}}\circ (f_{w_{-1}})^{-1}(x)) < \delta_{\theta(w_{-1})}= \delta_{\theta(\theta^{-1}(w))}=\delta_w.
		\]
		Again, for $k= 1$ with $f_{w_{-1}}(y) \in \Gamma_\delta (w,x)$ we have
		\[
		d(f_{w}(x), f_{w}(f_{w_{-1}}(y))) < \delta_{w_1},
		\]
		which corresponds to $k=2$ in $\Gamma_{\delta} (w_{-1}, (f_{w_{-1}})^{-1}(x))$. Indeed 
		\begin{align*}
			& d(f^{2}_{w_{-1}}(y), f^{2}_{w_{-1}}((f_{w_{-1}})^{-1}(x))) < \delta_{\theta^{2}(w_{-1})}\\ 
			& d(f_{\theta(w_{-1})}\circ f_{w_{-1}}(y), f_{\theta(w_{-1})}\circ f_{w_{-1}}\circ (f_{w_{-1}})^{-1}(x) ) < \delta_{\theta(w)}\\
			& d(f_w \circ f_{w_{-1}}(y), f_{w}(x)) < \delta_{w_1}.
		\end{align*}
		By last, for $k=-1$ with $f_{w_{-1}}(y) \in \Gamma_\delta (w,x)$ we have
		\[
		d((f_{w_{-1}})^{-1}(x), (f_{w_{-1}})^{-1}\circ f_{w_{-1}}(y)) = d((f_{w_{-1}})^{-1}(x), y) < \delta_{w_{-1}}.
		\]
		That corresponds $k=0$ in $\Gamma_{\delta} (w_{-1}, (f_{w_{-1}})^{-1}(x))$. Indeed 
		\[
		d((f_{w_{-1}})^{-1}(x), y) < \delta_{w_{-1}}.
		\]
		In this way, the sequence $(\delta_{w_{-1}})_{k\in \mathbb{Z}}$ in the fiber $w_{-1}$ for $\Gamma( (f_{w_{-1}})^{-1}(x), w_{-1})$ is given by $\delta_{w_k} = \delta_{(w_{-1})_{k+1}}$ for all $k\in \mathbb{Z}$.
	\end{proof}
	
For the next result we need the following definition: Let $M$ and $N$ be compact metric spaces
	and a homeomorphism $\varphi : M \to  N$. We denote by $\varphi_{\ast}:\mathcal{M}_{1}(M)\longrightarrow  \mathcal{M}_{1}(N)$  the pullback of $\mu\in \mathcal{M}_{1}(M)$
	defined by $\varphi_{\ast} (\mu)(A) = \mu (\varphi^{-1}(A))$ for all borelian $A \subset M$. 
	The topology we are going to consider on the space $\mathcal{M}_{1}(M)$ is the weak* topology, which is generated by the basis of neighborhoods, 
	$$
	V(\mu_0, \Phi, \epsilon)=\biggr\{\mu_{1}\in \mathcal{M}_{1}(M): \biggr|\int \varphi_{i} \  d\mu_0 -  \int \varphi_{i} \ d\mu_1 \biggl| < \epsilon \biggl\} \\
	$$
	where $\epsilon>0$ and $\Phi=\{\varphi_i:M\rightarrow \mathbb{R}\}$ is a set of continuous functions for $1\leq i \leq m$ and $m \in \mathbb{N}$. Similarly, we consider the weak* topology defined on $\mathcal{M}_{1}(N)$.
	
	Now, we observe that given $\epsilon >0$,   $\phi_{i}: N \longrightarrow \mathbb{R}$  continuous functions for $1\leq i \leq m$ and $\mu_0, \mu_1 \in \mathcal{M}_{1}(M)$, we have
	$$
	\biggr|\int \phi_{i} \  d \varphi_{\ast}\mu_0 -  \int \phi_{i} \ d  \varphi_{\ast}\mu_1 \biggl|\\
	= \biggr|\int (\phi_{i}\circ \varphi) \ d\mu_0 -  \int (\phi_{i}\circ \varphi) \ d\mu_1 \biggl| < \epsilon  \\
	$$
	for all \ $1\leq i \leq m$.  
	This means that $\varphi_{\ast}$ is continuous in the weak* topology.			

\begin{lemma}
	\label{lemma2thc}
	If $(\mu_w)_{w\in \Omega}$ is the decomposition of the expansive measure $\mu$ for a RDS $f=\{ f_w\}_w$, then  the decomposition $(f_{w_{-1}})_{*}\mu_{w_{-1}})_{w\in \Omega}$ is expansive for this same system.
\end{lemma}

\begin{proof}
	By the previous Lemma and taking that $(\mu_w)_{w\in \Omega}$ is expansive, we have
	\begin{align*}
		\mu_{w_{-1}} (\Gamma_{\delta} ((f_{w_{-1}})^{-1}(x)), w_{-1}) &= \mu_{w_{-1}} ((f_{w_{-1}})^{-1}(\Gamma_\delta (x,w))) \\
		&=(f_{w_{-1}})_{*}\mu_{w_{-1}}  (\Gamma_{\delta}(x,w)) \\
		&= 0.
	\end{align*}
\end{proof}
For $w \in \Omega$, we define the following sequence of measures for $n\geq 2$
\begin{align*}
	\mu_{n,w} &= \dfrac{1}{n} \biggl( \mu_w +\sum_{i=1}^{n-1}\mu_{w_{-i}}(f_{w_{-1}}\circ \cdots\circ f_{w_{-i}})^{-1}\biggr)\\
	& = \dfrac{1}{n}\biggl(\mu_w+\sum_{i=1}^{n-1} (f_{w_{-1}}\circ \cdots \circ f_{w_{-i}})_{*}\mu_{w_{-i}}\biggr),
\end{align*}
each $\mu_{n,w}$ is a probability measure on $\E_w$ compact, then due to compactness of $\mathcal{M}_{1}(M)$, there exists a subsequence $n_k \to \infty$ such that $\mu_{n_{k},w}$ converges to a Borel probability $\mu'_w$ in the weak* topology. Without loss of generality, let us assume that the subsequence is $\mu_{n,w}$. \\
To find a measurable and invariant decomposition along the fibers, we are going to consider two sequences of measures: one of them on $\E_w$ and other on $\E_{\theta(w)}=\E_{w_1}$.

\begin{lemma}
	\label{lemma3thc}
	If $\mu_{n,w}\to \mu'_{w}$ and $(f_{w})_{\ast}\mu_{n,w}\to \mu'_{w_1}$ in the weak* topology, then $\mu'_{w}((f_{w})^{-1}(A))= \mu'_{w_1}(A)$ for all measurable subset $A \subset \E_{w_1}$. Therefore $(\mu'_w)_{w\in \Omega}$ is measurable, invariant and expansive.
\end{lemma}

\begin{proof}
	Let $\{\phi_{i}: \E_{w_1} \longrightarrow \mathbb{R}\}_{1\leq i \leq m}$ be a family of continuous functions and $\epsilon >0$. So, considering a basis of neighborhoods in the weak* topology $V(\mu'_{w_1}, \{\phi_i\}_{1\leq i \leq m}, \epsilon/2)$ on $\mathcal{M}_{1}(\mathcal{E}_{w_1})$, we have
		
		$$
		\biggl|\int \phi_{i} \ d(f_w)_{\ast}\mu_{n,w} - \int \phi_i \ d\mu'_{w_1}  \biggr| < \dfrac{\epsilon}{2}.
		$$
		As $f_w$ is continuous, the family $\{\phi_i\circ f_w \}_{1\leq i\leq m}$ is also of continuous functions. Again, for the basis of neighborhoods in the weak* topology $V(\mu'_{w}, \{\phi_i\circ f_w\}_{1\leq i \leq m}, \epsilon/2)$ on $\mathcal{M}_{1}(\mathcal{E}_w)$, we have
		$$  
		\biggl|\int \phi_{i}\circ f_w \ d\mu_{n,w} - \int \phi_i\circ f_w \ d\mu'_{w}  \biggr| < \dfrac{\epsilon}{2}
		$$
		for every $1 \leq i \leq m$ and sufficiently large $n$. Therefore,
		
		\begin{align*}
			&\biggl|\int \phi_i \ d(f_w)_{\ast}\mu'_w - \int \phi_i \ d\mu'_{w_1}  \biggr| = \biggl|\int \phi_i\circ f_w \ d\mu'_w - \int \phi_i \ d\mu'_{w_1}  \biggr| \leq \\
			& \leq \biggl|\int \phi_i\circ f_w \ d\mu'_w - \int \phi_i \circ f_w \ d\mu_{n,w}  \biggr|+ \biggl|\int \phi_i\circ f_w \ d\mu_{n,w} - \int \phi_i \ d(f_w)_{\ast}\mu_{n,w}  \biggr| +\\
			&+ \biggl| \int \phi_i \ d(f_w)_{\ast}\mu_{n,w} - \int \phi_i \ d\mu'_{w_1}\biggr| < \epsilon,
		\end{align*}
		for every $1 \leq i \leq m$  and sufficiently large $n$. From here, we have
		$$
		(f_{w})_{\ast}\mu'_w = \mu'_{w_1}.
		$$
						
	Using the fact that the disintegration $(\mu_w)_{w\in \Omega}$ is measurable. Then we have the measurability of $w\mapsto \mu'_w$. Now let
	\[
	d\mu'(w,x):= d\mu'_w(x) d\mathbb{P}(w) 
	\]
	we obtain that $\mu'$ is a measurable expansive invariant probability measure of the system.

		\end{proof}

		

		\section{Proof of Theorem C}
		
		In this section, we will prove our third main result adapting some ideas from \cite{Artigue-Carrasco}.
		
		\begin{proof}[Proof of Theorem C]
			$(i)\to (ii)$  By hypothesis there exist $\Omega_0$ with $\mathbb{P}(\Omega_0)=1$   and a random variable $\delta=\delta(w)>0$ such that $\Gamma_{\delta}(x,w)$ is countable for all $w\in \Omega_0$  and $x\in X$. Let $\mu$ be a Borel probability measure on $\Omega \times X$ such that its factorization is non-atomic on $\Omega_0$. By $\sigma$-additivity we have that $\mu_w(\Gamma_{\delta}(x,w))=0$ for all $w\in \Omega_0$ and all $x \in X$. Therefore, $\mu$ is a random expansive measure for $f$.                                           
			
			\medskip
			
			$(ii)\to (i)$  Arguing by contradiction, we assume that $f$ is not random countably-expansive, that is, there exists a sequence of positive random variable  $\delta_{n}(w)\rightarrow 0$ and for each $n$, there exists a set $A_n\subset \Omega$, such that $\pp(A_n)>0$ and for each $w\in A_n$ there exists a point $x\in X$ such that $\Gamma_{\delta_n}(x,w)$ is uncountable. For each $n\ge 1$ we choose some $(w_n,x_n)$ as above. By using some arguments of \cite{Morales-Sirvent} (see Proposition 1.7 in \cite{Morales-Sirvent} and also Theorem 8.3  in \cite{Parthasarathy}), we can obtain a non-atomic Borel probability measure $\mu_{(w_n,x_n)}$ on $X$ such that $\mu_{(w_n,x_n)}(\Gamma_{\delta_n}(x_n , w_n))=1$. Thus,  defining  
			\[
			\mu(B) =\sum_{n=1}^{\infty} \frac{1}{2^{n}}\mu_{(w_n,x_n)}(B),
			\]
			where $B$ is a measurable subset of $X$, the measure $\mu$ is a non-atomic Borel probability on $X$.
			
			Consider the product measure $\nu = \mu \times \mathbb{P}$ on $\Omega\times X$, then we have that $\nu\in \mathcal{P}_{\pp}(\Omega\times X)$  and its factorization is non-atomic in each fiber and given by $\mu$ for all $w\in \Omega$ i.e., it is the same measure on each fiber.
			
			On the other hand, we have that $\nu$ is a random expansive measure for $f$ with characteristic of expansivity $\delta$, that is, 
			$\nu_w(\Gamma_{\delta}(x,w))=\mu(\Gamma_{\delta}(x,w))=0$ for $\pp$-a.e. $w$ and $x\in X$. In particular, we have $\mu(\Gamma_{\delta}(x_n,w_n)) = 0$. 
			
			Consider $n^{\ast}\in \mathbb{N}$ such that $\delta_{n} < \delta$, whenever  $n>n^{\ast}$. Thus, we obtain $\Gamma_{\delta_{n}}(x_{n},w_n) \subset \Gamma_{\delta}(x_{n},w_n)$, which implies 
			\[\mu(\Gamma_{\delta}(x_{n},w_n)) \geq \mu(\Gamma_{\delta_{n}}(x_{n},w_n)) \geq (1/2^{n})\mu_{(w_n,x_n)}(\Gamma_{\delta_{n}}(x_{n},w_n)) > 0
			\]
			and we get a contradiction.
		\end{proof}


		\section{Examples}\label{Examples}

		In this section we present some examples related to the notions of expansivity introduced in the paper.

		The first example consists of a random transformation that has an expansive measure in the random sense.

		\begin{example}
			Consider an abstract dynamical system $(\Omega, \mathcal{F}, \mathbb{P}, \theta)$ and the compact metric space $X = \bar{\N}^{\N}$ where $\bar{\N}$ is the one point compactification of $\N$ and $X$ is endowed with the product of the discrete topology on $\N$. A metric on $X$ is defined by 
			\[
			d(x,y)  =  \sum^{+\infty}_{i=0} 2^{-i} \left|\frac{1}{x_i} - \frac{1}{y_i}\right|.
			\]
			The (left-) shift $\sigma$  on $X$ defined by $\sigma (\{x_n\}) = \{x_{n+1}\}$ is continuous.
			
			For a random variable $k:\Omega \to \N$  that satisfies $k(w)\geq 2$ for all $w\in \Omega$,  we put 
			\[
			\Sigma_{k}^{+}(w) : = \{x\in X ; \, x_i \leq k(\theta^{i}(w))  \  \  \textrm{for all} \  \ i\in \N \}.
			\]
			It was shown \cite{BG}  that $\{\sigma: \Sigma_{k}^{+}(w) \to \Sigma_{k}^{+}(\theta(w))\}$ defines a random dynamical system on the compact bundle $\Sigma_{k}^{+} = \{(w,x) ; \, x \in \Sigma_{k}^{+}(w)\}$.
			
			Given $x=\{x_n\}\in X$ we denote by 
			\[C_{w}(x_0  x_1  ...  x_{n-1}) = \{ y=\{y_n\}\in X: x_i = y_i \ \textrm{for all} \    0 \leq i \leq n-1 \}
			\]
			the $(n,w)$-cylinder that contains $x$. The $(n,w)$-cylinders form a sigma-algebra on $\Sigma^{+}_{k}(w)$. For each  $w\in \Omega $ we define the measure on the $(w,n)$-cylinder by
			\begin{equation}\label{eq1}
				\tilde{\mu}_{w} (C_{w}(x_0  x_1  ...  x_{n-1})) = \prod_{i=0}^{n-1} p^{l}_{\theta^{i}(w)}(\{x_i\}) = \dfrac{1}{k(w)\cdot ... \cdot k(\theta^{n-1}(w))}
			\end{equation}
			for random probability vectors 
			$\{p(w) = (p^1_w, ... , p^{k(w)}_w)\}$, such that $
			p_{w}^{l}(\{j\}) = \dfrac{1}{k(w)}
			$ for all $l\in \{1,...,k(w)\}$.  Then, by the Kolmogorov extension theorem, for $\mathbb{P}$-a.s. $w\in \Omega$ there exists a unique probability measure $\mu_w$ on $\Sigma_{k}^{+}(w)$ that satisfies \eqref{eq1}. We
			define the probability measure on $\Sigma_{k}^{+}$ by $\mu (A) = \int \mu_w(A(w)) d\mathbb{P}(w)$ for $A\subset \Sigma_{k}^{+}$.
			In \cite{Kifer2000} the authors showed that the random shift is random expansive with characteristic of expansivity $k^{-2}(w)$. We will show that the measure $\mu$ is random expansive for the random shift with characteristic of expansivity $k^{-2}(w)$. In fact, given $w\in \Omega$ if there exists some $y \in \Sigma_{k}^{+}(w)$ such that 
			$\mu_w(\Gamma_{k^{-2}(w)}(w,y)) > 0$, by the random expansivity of random shift, there exists some $m_{w}$ such that 
			$\mu_{w} (\{y_w\}) > \dfrac{1}{2^{m_w}}$. On the other hand, for $n>m_{w}$ the measure of any $(w,n)$-cylinder that contains $y_w$ is given by 
			\[
			\mu_{w} (\{y_w\}) < \mu_{w}(C_{w}(y_0  y_1  ...  y_{n-1})) = \dfrac{1}{k(w)\cdot ... \cdot k(\theta^{n-1}(w))} < \dfrac{1}{2^{n}}.
			\]
			This is an absurd. Therefore,
			the measure $\mu_{w}$ is non atomic for all $w\in \Omega$ and this shows that $\sigma$ is random $\mu$-expansive.\end{example}

		In the following, we show a RDS with no random expansive measure.
		
		\begin{example}
			Denote by $\mathcal{F}$ the class of isometries $f:X \to X$, where $X$ is a separable metric space. Let $\theta:\Omega \to \Omega$ be any invertible function preserving an ergodic measure $\mathbb{P}$ on $\Omega$. 
			
			Consider the random dynamical system $f = \{f_w\}_{w}$. 
			We claim that there is no random expansive measure for $f$.  In fact, let us assume that there exists some Borel probability  $\mu$ on $\Omega \times X$ that is random expansive for $f$ with characteristic of expansivity $\delta=\delta(w)>0$. 
			
			Fixed $w\in \Omega$ and for all $x\in X$ we have $\Gamma_{\delta}(w,x) = B[x,\delta(w)]$,
			where  $B[x,\delta(w)]$ denotes the closed $\delta(w)$-ball around $x$.
			By the expansivity of $\mu$, for $\mathbb{P}$-a.e $w\in \Omega$ and all $x\in X$, we have
			\[
			\mu_w(\Gamma_{\delta}(w,x)) = \mu_{w}(B[x,\delta(w)])=0.
			\] 
			Since X is separable, we can take a countable covering $\{C_{1,w},..., C_{n,w},...\}$ of $X$ by closed subsets such that for each $n$ there exists a point $x_n\in X$ satisfying  $C_{n,w}\subset B[x_n,\delta(w)]$. From this we obtain that for $\mathbb{P}$-a.e $w\in \Omega$ 
			\[
			\mu_w(X) = \mu_w \left( \bigcup_{n}^{+\infty} C_{n,w} \right) \leq \sum_{n=1}^{+\infty} \mu_w  (B[x_n,\delta(w)])=0
			\]
			which is a contradiction, as desired.
		\end{example}
		As an application of Theorem A, we also show a system with positive relative entropy.
		\begin{example}
			\label{exemploentropositiva1}
			Let $(\Omega, \mathcal{B}, m, \theta)$ be a measure  preserving a dynamical system with an invertible and ergodic map $\theta: \Omega \rightarrow \Omega$ and compact metric spaces $(\mathcal{E}_w, d_w)$, $w\in \Omega$, normalized in size by $diam_{d_w}(\mathcal{E}_w) \leq 1$. Let $f=\{ f_w\}$ be a RDS generated by continuous, open and  surjective  maps $f_w: \mathcal{E}_w \rightarrow \mathcal{E}_{\theta(w)}$   together with functions $\eta:\Omega \rightarrow \mathbb{R}_+$, $w\mapsto \eta_w$ and a real number $\xi > 0$ satisfying the following conditions:
			\begin{itemize}
				\item [(H1)] $f_w(B_w(y,\eta_w)) \supset B_{\theta(w)}(f_w(y),\xi)$ for every $(w,y)\in \mathcal{E}$
				\item [(H2)] There exists a measurable function $\gamma:\Omega\rightarrow (1, +\infty )$, $w\mapsto \gamma_w$ such that, for $m$-a.e. $w\in \Omega$,
				\[
				d_{\theta(w)}(f_w(y_1),f_w(y_2)) \geq \gamma_w d_w(y_1, y_2) \quad \text{whenever} \quad d_w(y_1,y_2)< \eta_w, \quad y_1, y_2 \in \mathcal{E}_w
				\]
				\item [(H3)] The map $w\mapsto deg(f_w):= \sup_{y\in \mathcal{E}_{\theta(w)}} \sharp f^{-1}_{w}(\{y\})$ is measurable. 
				\item [(H4)] $\inf_{w\in \Omega}\gamma_w > 1$, \quad $ \sup_{w\in \Omega} deg(f_w) < \infty$ \quad \text{and} \quad $\sup_{w\in \Omega} \eta_{\xi}(w) < \infty$.
			\end{itemize}
			Adding other technical hypotheses, the authors in \cite{Urbanski} showed the existence of equilibrium states for a class of H\"older potentials. In the particular case of the null potential, it was shown that $h_\mu(f) = \int_{\Omega} \log deg(f_w) dm(w) >0$, where $\mu \in \mathcal{M}(\E, f)  \subset \mathcal{P}_{m}(\E)$. Therefore, by the Theorem A, there exists some measure which is positively expansive in the random sense.
		\end{example}
		
		In the following, we give an example of a random continuum-wise expansive system. 
		\begin{example}
			Let $X$ be a metric space and $f_0, f_1 : X \to X$ maps such that $f_0$ is an isometric map and $f_1$ is a deterministic continuum-wise expansive homeomorphism with constant of expansivity $\delta>0$. Consider the set $\{0,1\}^{\mathbb{Z}}$ and let $\sigma$ be the shift map on $\{0,1\}^{\mathbb{Z}}$. Define the map
			\[
			F:\{0,1\}^{\mathbb{Z}} \times X \to \{0,1\}^{\mathbb{Z}}\times X   \   \   \     \   \    F(w,x) = (\sigma(w),f_{w_0}(x)).
			\]
			The $n$-th power of this map is given by:
			\[
			F^{n}(w,x) = (\sigma^n(w), f_{w_{n-1}}\circ ... \circ f_{w_0}(x)).
			\]
			Then $F$ is a random dynamical system. Let $\mathbb{P}$ be an ergodic measure such that
			\[
			\mathbb{P}(\{w\in \Omega,  \,\, \textrm{there is a finite number of symbols} \, w_i=1\})=0
			\]
			(or $\mathbb{P}(A)=0$ for all countable set $A \in \{0,1\}^{\mathbb{Z}}$). Note that, given $w\in \Omega$ and a continuum set $C \subset X$ that is not a singleton, there exists $n=n(w,C)$ such that $\text{diam}(f_{w}^{n}(C)) > \delta$.
		\end{example}

		\bibliographystyle{amsplain}

	\end{document}